\documentclass{elsarticle}

\usepackage{lineno,hyperref}

\modulolinenumbers[5]

\journal{Computers and Mathematics with Applications}

\usepackage{graphicx}
\usepackage{mathrsfs,color}
\usepackage{mathptmx, amsmath, amsfonts}
\usepackage[utf8]{inputenc}
\newtheorem{theorem}{Theorem}
\newtheorem{corollary}[theorem]{Corollary}
\newtheorem{lemma}[theorem]{Lemma}

\newcommand{\R}{\mathbb{R}}
\newcommand{\FI}{\mathcal{I}}
\newcommand{\iprod}[1]{\langle#1\rangle}
\newcommand{\bigiprod}[1]{\bigl\langle#1\bigr\rangle}
\newcommand{\biggiprod}[1]{\biggl\langle#1\biggr\rangle}
\newcommand{\Mult}{\mathcal{M}}

\newcommand{\Q}{\mathcal{Q}}
\newcommand{\E}{\mathcal E}

\newcommand{\B}{\mathcal B}

\begin{document}
\begin{frontmatter}
\title{Regularity theory for time-fractional advection-diffusion-reaction 
equations \tnoteref{mytitlenote}}
\tnotetext[mytitlenote]{The authors thank the University of New South Wales 
(Faculty Research Grant ``Efficient numerical simulation of anomalous transport 
phenomena''), the King Fahd University of Petroleum and Minerals (project 
No.~KAUST005) and the King Abdullah University of Science and Technology.}
%% Group authors per affiliation:
\author{William McLean}
\address{School of Mathematics and Statistics, The University of New South 
Wales, Sydney 2052, Australia}
\author{Kassem Mustapha,  Raed  Ali}
\address{Department of Mathematics and Statistics, KFUPM,
Dhahran, 31261, KSA}
\author{Omar  M. Knio}
\address{Computer, Electrical, Mathermatical Sciences and Engineering Division,
KAUST, Thuwal 23955, KSA}
\begin{abstract}
We investigate the behavior of the time derivatives of the solution to a linear 
time-fractional, advection-diffusion-reaction equation, allowing space- and 
time-dependent coefficients as well as initial data that may have low 
regularity. Our focus is on proving estimates that are needed for the error 
analysis of numerical methods. The nonlocal nature of the fractional derivative 
creates substantial difficulties compared with the case of a classical parabolic 
PDE. In our analysis,  we rely on  novel energy methods in combination with a 
fractional Gronwall inequality and certain properties of  fractional integrals.
\end{abstract}

\end{frontmatter}
%%%%%%%%%%%%%%%%%%%%%%%%%%%%%%%%%%%%%%%%%%%%%%%%%%%%%%%%%%%%%%%%%%%%%%%%%%%%%%%
\section{Introduction}
This paper is the sequel to a study~\cite{McLeanEtAl2019} of existence and 
uniqueness of the weak solution to a time-fractional PDE of the form
\begin{equation}\label{eq: FPDE}
\partial_tu-\nabla\cdot\bigl(\kappa\nabla\partial_t^{1-\alpha}
	-\vec F\partial_t^{1-\alpha}-\vec G\bigr)u
	+(a\partial_t^{1-\alpha}+b)u=g
\end{equation}
for $x\in\Omega$ and $0<t\le T$, subject to the boundary and initial conditions
\begin{align}
u(x,t)&=0&&\text{for $x\in\partial\Omega$ and $0\le t\le T$,} 
	\label{eq: Dirichlet bc}\\
u(x,0)&=u_0(x)&&\text{for $x\in\Omega$.}\label{eq: initial condition}
\end{align}
Various special cases of this problem occur in descriptions of subdiffusive 
transport, with the parameter~$\alpha$ arising from a continuous-time, random 
walk model~\cite[section~3.4]{KlafterSokolov2011} in which the waiting-time 
distribution is a power law decaying like $t^{-1-\alpha}$ as~$t\to\infty$
\cite{AngstmannEtAl2017,HenryLanglandsStraka2010,HenryWearne2000,
LanglandsHenryWearne2011,LiuAnhTurnerZhuang2003, MetzlerBarkaiKlafter1999}.  For 
more details, see our related paper~\cite{McLeanEtAl2019}.

Our purpose here is to derive estimates for the derivatives of~$u$, motivated 
by their crucial role in the error analysis of numerical 
methods~\cite{CuestaLubichPalencia2006,JinLiZhou2018,
KaraaPani2018,LeMcLeanMustapha2016,LeMcLeanMustapha2018,LiaoLiZhang2018,
Mustapha2015,StynesEtAl2017} for applications included in the 
class~\eqref{eq: FPDE} of time-fractional problems. For the basic fractional 
diffusion equation, given by the special case $\vec F=\vec G=\vec 0$ and 
$a=b=0$, the solution admits a series representation via separation of 
variables, which, in combination with the asymptotics of the Mittag--Leffler 
function, yields bounds on the time derivatives of~$u$ in 
various spatial norms~\cite{McLean2010,LiXie2017}.  One may also represent the 
solution in terms of a fractional 
resolvent~\cite{McLeanThomee2010,Fan2014,KeyantuoEtAl2016}. 
These simple approaches no longer work in the general case, and the analysis 
that follows relies instead on the tools used in our study~\cite{McLeanEtAl2019} 
of well-posedness: energy methods and a fractional Gronwall inequality.

We assume that $0<\alpha<1$ and that the spatial 
domain~$\Omega\subseteq\R^d$ ($d\ge1$) is bounded and Lipschitz. The
coefficients $\vec F$, $\vec G$, $a$~and $b$, as well as the source 
term~$g$, may depend on $x$~and $t$, but the generalized 
diffusivity~$\kappa=\kappa(x)$ may depend only on~$x$.  Our theory requires 
that for appropriate~$m\ge1$,
\begin{equation}\label{eq: reg coeff} 
\vec F, \vec G\in C^{m+1}\bigl([0,T];W^1_\infty(\Omega)^d\bigr)
\quad\text{and}\quad
a, b\in C^m\bigl([0,T];L_\infty(\Omega)\bigr),
\end{equation}
where $W^k_p(\Omega)$ denotes the Sobolev space of functions with all partial 
derivatives up to and including order~$k$ belonging to~$L_p(\Omega)$.  The 
generalized diffusivity is permitted to be a bounded, 
$d\times d$~matrix-valued function, that is,
$\kappa\in L_\infty(\Omega;\R^{d\times d})$.  In addition, we ensure that the 
spatial operator~$v\mapsto-\nabla\cdot(\kappa\nabla v)$ is uniformly elliptic 
on~$\Omega$ by assuming $\kappa(x)$ is symmetric and positive-definite with its 
minimal eigenvalue bounded away from zero.  
The fractional time derivative is understood in the 
Riemann--Liouville sense, that is 
$\partial_t^{1-\alpha}v(x,t)=\partial_t \FI^\alpha v(x,t)$ with~$\FI^\alpha$ 
the fractional integral given by
\[
\FI^\alpha v(x,t)=\int_0^t\omega_\alpha(t-s)v(x,s)\,ds
\quad\text{where}\quad
\omega_\alpha(t)=t^{\alpha-1}/\Gamma(\alpha).
\]

Let $\iprod{\cdot,\cdot}$ denote the inner product in $L_2(\Omega)$~or 
$L_2(\Omega)^d$.  The weak solution~$u$ of~\eqref{eq: FPDE} is defined by the 
condition
\begin{equation}\label{eq: weak u}
\iprod{u(t),v}+\bigiprod{(\kappa\FI^\alpha\nabla u)(t),\nabla v}
    -\bigiprod{(\vec B_1u)(t),\nabla v}+\iprod{(B_2u)(t),v}=\iprod{f(t),v},
\end{equation}
for all $v\in H^1_0(\Omega)$, where $f(t)=u_0+\FI^1g(t)$ and
\begin{equation}\label{eq: vec B1 B2}
\begin{aligned}
\vec B_1\phi(t)&=\FI^1(\vec F\partial_t^{1-\alpha}\phi)(t)
    +\FI^1(\vec G\,\phi)(t),\\
B_2\phi(t)&=\FI^1(a\partial_t^{1-\alpha}\phi)(t)
    +\FI^1(b\,\phi)(t).
\end{aligned}
\end{equation}
To see why, take the inner product of~\eqref{eq: FPDE} with~$v$, apply the 
first Green identity and integrate in time (making use of the initial 
condition).  We proved in our previous 
paper~\cite[Theorems 4.1 and 4.2]{McLeanEtAl2019} that the above problem is 
well-posed in the following sense. 

\begin{theorem}\label{thm: existence} 
Assume that the coefficients satisfy~\eqref{eq: reg coeff} for~$m=1$, 
that the source term $g$ satisfies $\|g(t)\|\le Mt^{\eta-1}$ for $0<t\le T$,
where $M$~and $\eta$ are positive constants, and that the initial data 
$u_0\in L_2(\Omega)$.  Then,
problem~\eqref{eq: FPDE}--\eqref{eq: initial condition} has a unique weak 
solution~$u \in L_2\bigl((0,T);L_2(\Omega)\bigr)$ that satisfies 
\eqref{eq: weak u}, and is such that
\begin{enumerate}
\item The restriction~$u:(0,T]\to L_2(\Omega)$ is continuous.
\item If $0<t\le T$, then
$u(t)\in H^1_0(\Omega)$ with $\|u(t)\|+t^{\alpha/2}\|\nabla u(t)\|
\le C\bigl(\|u_0\|+Mt^\eta\bigr).$
\item $\FI^\alpha u,\,B_2u \in C([0,T];L_2(\Omega))$ and 
$\FI^\alpha\nabla u,\,\vec
B_1u\in C([0,T];L_2(\Omega)^d)$.
\item If $t=0$, then $\FI^\alpha u=B_2u=0$, $\FI^\alpha\nabla u=\vec B_1u=0$
and $u(0)=u_0$.
\item If $t\to0$, then $\iprod{u(t),v}\to\iprod{u(0),v}$ for each 
$v\in L_2(\Omega)$. 
\end{enumerate}
\end{theorem}

Regarding part~1 and the weak continuity in part 5, we will show 
in Theorem~\ref{thm: u cts} that $u$ is continuous on the closed 
interval~$[0,T]$ provided $u_0\in\dot H^\mu(\Omega)$ for some~$\mu>0$.  

Some partial results on the regularity of~$u$ are known.
If the coefficients in~\eqref{eq: FPDE} are independent of time, and if 
$\vec G=\vec 0$~and $b=0$, then by applying $\FI^{1-\alpha}$ to both sides the 
fractional PDE may be written in the alternative form
\begin{equation}\label{eq: alt FPDE}
\partial_t^\alpha u-\nabla\cdot(\kappa\nabla u-\vec F u)+au=\FI^{1-\alpha}g.
\end{equation}
Sakamoto and Yamamoto~\cite[Corollary~2.6]{SakamotoYamamoto2011} show, for 
example, that if $g=0$ then the solution of~\eqref{eq: alt FPDE} satisfies a 
bound of the form $\|\partial_t^mu\|\le Ct^{-m}\|u_0\|$, where $\|\cdot\|$ 
denotes the norm in~$L_2(\Omega)$.  Mu, Ahmad and Huang~\cite{MuEtAl2017} 
obtain analogous estimates using weighted H\"older norms. Recently, Le et 
al.~\cite{LeMcLeanStynes2018} studied \eqref{eq: FPDE} for the case 
$\vec G=\vec 0$~and $a=b=0$, with $\vec F=\vec F(x,t)$.  One of their 
regularity results~\cite[Theorem~7.3]{LeMcLeanStynes2018} gives the bound
$\|\partial_t^mu\|\le Ct^{-m+1/2}\|u_0\|_{H^2(\Omega)}$ when $g=0$, 
subject to the restriction $1/2<\alpha<1$.

The next section gathers together some technical preliminaries needed for our 
analysis, which uses delicate energy arguments, a fractional Gronwall 
inequality and several properties of fractional integrals to prove 
\emph{a priori} estimates for the weak solution~$u$ of
\eqref{eq: FPDE}--\eqref{eq: initial condition}. In 
Section~\ref{sec: regularity}, we estimate the derivatives of $u$~and $\nabla u$ 
with respect to time assuming $u_0\in L_2(\Omega)$. For example, 
Corollary~\ref{cor: reg simple} shows that if 
$g(t)\equiv0$ then, with~$m\ge1$ such that~\eqref{eq: reg coeff} holds,
\[
\|\partial_t^mu\|+t^{\alpha/2}\|\partial_t^m\nabla u\|
	+t^{-\alpha}\|\partial_t^{m-\alpha}u\|
	+t^{-\alpha/2}\|\partial_t^{m-\alpha}\nabla u\|
	\le Ct^{-m}\|u_0\|
\]
for~$0<t\le T$. Unlike a classical parabolic PDE, the fractional 
problem~\eqref{eq: FPDE} exhibits only limited spatial smoothing as~$t$ 
increases~\cite{McLean2010}, and in Section~\ref{sec: H2}  we investigate the 
consequences of more regular initial data. For example, 
Theorems \ref{thm: u cts}~and \ref{thm: H2 reg} show that when $g(t)\equiv0$~and 
$u_0\in\dot H^\mu(\Omega)$ for~$0\le\mu\le2$, and under additional assumptions 
on $\kappa$~and $\Omega$,
\[
\|\partial_t^mu\|
	+t^{-\alpha}\|\partial_t^{m-\alpha}u\|
	+t^\alpha\|\partial_t^mu\|_{H^2(\Omega)}
	+t^{-\alpha/2}\|\partial_t^{m-\alpha}\nabla u\|
	\le Ct^{-m+\alpha\mu/2}\|u_0\|_\mu.
\]
The paper concludes with an Appendix containing three technical lemmas.

%%%%%%%%%%%%%%%%%%%%%%%%%%%%%%%%%%%%%%%%%%%%%%%%%%%%%%%%%%%%%%%%%%%%%%%%%%%%%%%
\section{Preliminaries and notations}\label{sec: prelim}
This section introduces some notations and states some technical results that 
will be used in our subsequent regularity analysis. As in our recent 
paper~\cite{McLeanEtAl2019}, we define the quadratic operators 
$\Q^\mu_1$ and $\Q^\mu_2$, for $\mu\ge0$ and $0\le t\le T$, by
\begin{equation}\label{eq: D1D2}
\Q^\mu_1(\phi,t)=\int_0^t\iprod{\phi,\FI^\mu\phi}\,ds
\quad\text{and}\quad
\Q^\mu_2(\phi,t)=\int_0^t\|\FI^\mu\phi\|^2\,ds.
\end{equation}
These operators coincide when~$\mu=0$, so we write $\Q^0=\Q_1^0=\Q_2^0$. 
We recall the following positivity property \cite[Theorem~2]{NohelShea1976} 
\begin{equation}\label{eq: Q1 Plancherel}
\Q^\mu_1(\phi,T)\ge0\quad\text{for $0\le \mu\le 1$.}
\end{equation}

The next four lemmas establish key inequalities satisfied by $\Q^\mu_1$~and
$\Q^\mu_2$.

\begin{lemma}[\protect{\cite[Lemma 3.2]{LeMcLeanMustapha2018}}]
\label{lem: alpha dep}
If $0<\alpha<1$ and $\epsilon>0$, then
\begin{gather}
\biggl|\int_0^t\iprod{\phi,\FI^\alpha\psi}\,ds\biggr|
	\le\frac{\Q^\alpha_1(\phi,t)}{4\epsilon(1-\alpha)^2}
	+\epsilon\,\Q_1^\alpha(\psi,t), \label{eq: A}\\
\Q^\alpha_2(\phi,t)\le\frac{2t^\alpha}{1-\alpha}\,
	\Q^\alpha_1(\phi,t), \label{eq: B}\\
\Q^\alpha_1(\phi,t)\le2t^\alpha\,\Q^0(\phi,t).\label{eq: C}
%\biggl|\int_0^t\iprod{\phi,\FI^\alpha\psi}\,ds\biggr|
%        \le\frac{t^\alpha\Q^0(\phi,t)}{2\epsilon(1-\alpha)^2}
%        +\epsilon\,\Q^\alpha_1(\psi,t). \label{eq: AC}
\end{gather}
\end{lemma}

\begin{lemma}[\protect{\cite[Lemmas 2.2 and 2.3]{McLeanEtAl2019}}]
\label{lem: D}
If $0<\alpha\le1$, then for $\phi \in L_2\bigl((0,t),L_2(\Omega)\bigr)$,
\[
\Q^\alpha_2(\phi,t)\le2\int_0^t \omega_\alpha(t-s)\Q^\alpha_1(\phi,s)\,ds
\quad\text{and}\quad
\FI^{1-\alpha}\bigl(\|\FI^\alpha\phi\|^2\bigr)\le2\Q^\alpha_1(\phi,t).
\]
Furthermore, if $\phi\in W^1_1\bigl((0,t);L_2(\Omega)\bigr)$~and
$\phi(0)=\FI^\alpha\phi'(0)=0$, then
$$\|\phi(t)\|^2\le2\omega_{2-\alpha}(t)\,\Q_1^\alpha(\phi',t).$$
\end{lemma}

\begin{lemma}[\protect{\cite[Lemma 3.1]{LeMcLeanMustapha2018}}]
\label{lem: E} 
If $0\le\mu\le\nu\le1$, then
$\Q^\nu_2(\phi,t)\le 2t^{2(\nu-\mu)}\Q^\mu_2(\phi,t)$.
\end{lemma}

We will make essential use of the following fractional Gronwall inequality.

\begin{lemma}[\protect{\cite[Theorem~3.1]{DixonMcKee1986}}]
\label{lem: Gronwall}
Let $\beta>0$ and $T>0$.  Assume that $\mathsf{a}$ and $\mathsf{b}$ are
non-negative and non-decreasing functions on the interval~$[0,T]$.  If
$\mathsf{q}:[0,T]\to\R$ is an integrable function satisfying
\[
0\le\mathsf{q}(t)\le\mathsf{a}(t)
	+\mathsf{b}(t)\FI^\beta{q}(t)
    \quad\text{for $0\le t\le T$,}
\]
then
\[
\mathsf{q}(t)\le\mathsf{a}(t)E_\beta\bigl(\mathsf{b}(t)t^\beta\bigr)
    \quad\text{for $0\le t\le T$.}
\]
\end{lemma}

Let $\Mult^j$ denote the operator of pointwise multiplication by~$t^j$, that 
is, 
\[
(\Mult^j\phi)(t)=t^j\phi(t), 
\]
and note the commutator properties (for any integer~$j\ge1$ and any 
real~$\mu\ge0$)
\begin{equation}\label{eq: commutator}
\partial_t^j\Mult-\Mult\partial_t^j=j\partial_t^{j-1},\quad
\partial_t\Mult^j-\Mult^j\partial_t=j\Mult^{j-1},\quad
\Mult\FI^\mu-\FI^\mu\Mult=\mu\FI^{\mu+1}.
\end{equation}
The following identities then follow by induction on~$m$.

\begin{lemma}\label{lem: m-fold}
For $0\le q\le m$ and $\mu\ge0$, there exist constant coefficients $a^{m,q}_j$,
$b^{m,q}_j$, $c^{m,\mu}$~and $d^{m,\mu}_j$ such that
\begin{align}
\partial_t^q\Mult^m&=\Mult^m\partial_t^q
	+\sum_{j=1}^qa^{m,q}_j\Mult^{m-j}\partial_t^{q-j},\label{eq: m-fold 1}\\
\Mult^m\partial_t^q&=\partial_t^q\Mult^m
	+\sum_{j=1}^qb^{m,q}_j\partial_t^{q-j}\Mult^{m-j},\label{eq: m-fold 2}\\
\FI^\mu\Mult^m&=\Mult^m\FI^\mu
	+\sum_{j=1}^mc^{m,\mu}_j\Mult^{m-j}\FI^{\mu+j},\label{eq: m-fold 3}\\
\Mult^m\FI^\mu&=\FI^\mu\Mult^m
	+\sum_{j=1}^md^{m,\mu}_j\FI^{\mu+j}\Mult^{m-j}.\label{eq: m-fold 4}
\end{align}
For later reference, we set $a^{m,q}_0=b^{m,q}_0=c^{m,\mu}_0=d^{m,\mu}_0=1$ and
\[
\tilde a^{m,q}_j=a^{m,q}_{q-j},\quad
\tilde b^{m,q}_j=b^{m,q}_{q-j},\quad
\tilde c^{m,\mu}_j=c^{m,\mu}_{m-j},\quad
\tilde d^{m,\mu}_j=d^{m,\mu}_{m-j}.
\]
\end{lemma}

When $\mu=0$ the formulas involving $\FI^\mu$ become redundant, and we see that
$c^{m,0}_j=0=d^{m,0}_j$ for $1\le j\le m$. Likewise, $\omega_{-j}(t)=0$ 
for $0\le j\le m$ since $\Gamma(z)$ has a pole at~$z=-j$.
We conclude this section by noting that if $m\ge1$~and $\mu\ge0$, then
\begin{equation}\label{eq: FI omega}
\partial_t^m\FI^\mu\phi(t)=\FI^\mu\partial_t^m\phi(t)
	+\sum_{j=0}^{m-1}(\partial_t^j\phi)(0)\omega_{\mu-j}(t)
\quad\text{for $\phi\in W^m_1\bigl((0,t);L_2(\Omega)\bigr)$},
\end{equation}
which amounts to a restatement of the relation between the Riemann--Liouville 
and Caputo fractional derivatives. 
%%%%%%%%%%%%%%%%%%%%%%%%%%%%%%%%%%%%%%%%%%%%%%%%%%%%%%%%%%%%%%%%%%%%%%%%%
\section{Regularity of the weak solution}\label{sec: regularity}
Our aim in this section is to estimate 
higher-order time derivatives of~$u$ assuming appropriate bounds on
the higher-order time derivatives of~$f$ (and hence, ultimately, of~$g$),
as well as sufficient smoothness of the coefficients in~\eqref{eq: FPDE}.  We
will not attempt to prove the \emph{existence} of the higher-order derivatives
of~$u$, which could be done by estimating the corresponding derivatives of the 
projected solution~$u_X$ from our earlier paper~\cite{McLeanEtAl2019} 
corresponding to a finite dimensional subspace~$X=X_n\subset H^1_0(\Omega)$, 
and then taking appropriate limits as~$n\to\infty$. For the remainder of the 
paper, we assume that \eqref{eq: reg coeff} holds, and that 
\[
\|g^{(j-1)}(t)\|=O(t^{\alpha-j})\quad\text{as $t\to0$, for~$1\le j\le m$.}
\] 
It follows that the existence and uniqueness of the weak solution~$u$ are 
guaranteed by Theorem~\ref{thm: existence}. Henceforth, $C$ will denote a 
generic constant that may depend on the coefficients in~\eqref{eq: FPDE}, the 
spatial domain~$\Omega$, the time interval~$[0,T]$, the fractional 
exponent~$\alpha$, the parameter~$\eta$, and the integer~$m$ 
in~\eqref{eq: reg coeff}.  Also, we rescale the time variable if necessary so 
that the minimum eigenvalue of~$\kappa$ satisfies
\begin{equation}\label{eq: min eig}
\lambda_{\min}\bigl(\kappa(x)\bigr)\ge1\quad\text{for $x\in\Omega$.}
\end{equation}

For brevity, we introduce some more notations. Let 
\begin{equation}\label{eq: Bmu}(B^\mu_\psi\phi)(t)=\psi(t)\,\FI^\mu\phi(t)
	-\FI^1(\psi'\,\FI^\mu\phi)(t),\quad\text{for $0\le \mu \le 1$.}
\end{equation}
Integrating by parts and recalling~\eqref{eq: vec B1 B2}, we find that 
\begin{equation}\label{eq: B1 B2}
\vec B_1=B_{\vec F}^\alpha+B_{\vec G}^1
\quad\text{and}\quad
B_2=B_a^\alpha+B_b^1.
\end{equation}
Generalizing \eqref{eq: Bmu}, for $j\in\{0,1,2,\ldots\}$ we put
\[
B^{\mu,j}_\psi\phi(t)
	=\partial_t^j\Big(\Mult^j\FI^1(\psi\partial_t^{1-\mu}\phi)\Big)(t)
	=(\Mult^jB^\mu_\psi\phi)^{(j)}(t),
\]
and generalizing \eqref{eq: D1D2} we put
\[
\Q^{\mu,j}_i(\phi,t)=\Q^\mu_i\bigl((\Mult^j\phi)^{(j)},t\bigr),
\quad\text{for $0\le t\le T$ and $i\in\{1,2\}$,}
\]
with $\Q^{0,j}=\Q^{0,j}_1=\Q^{0,j}_2$.   The next result relies on Lemma 
\ref{lem: B integrals} from the Appendix.

\begin{lemma}\label{lem: stability C}
For $0< t\le T$ and for $m\ge1$,
\[
\Q^{\alpha,m}_1(u,t) +\Q^{\alpha,m}_2(\nabla u,t)
	\le Ct^\alpha\sum_{j=0}^m\Q^{0,j}(f,t),
\]
and
\[
\Q^{0,m}(u,t)+\Q^{\alpha,m}_1(\nabla u,t)
	\le C\sum_{j=0}^m\Q^{0,j}(f,t).
\]
\end{lemma}
\begin{proof}
Since $(\FI^\alpha\nabla u)(0)=0$ by part~4 of Theorem~\ref{thm: existence},
\begin{align*}
\int_0^t\bigiprod{\kappa\nabla\partial_s^{1-\alpha}u(s),\nabla v}\,ds
	&=\biggiprod{\kappa\int_0^t(\FI^\alpha\nabla u)'(s)\,ds,\nabla v}
	=\bigiprod{\kappa(\FI^\alpha\nabla u)(t),\nabla v},
\end{align*}
and by the identity in~\eqref{eq: m-fold 4},
\[
\Mult^m\FI^\alpha\nabla u=\FI^\alpha\Mult^m\nabla u
	+\sum_{j=0}^{m-1}\tilde d^{m,\alpha}_j\FI^{\alpha+m-j}\Mult^j\nabla u.
\]
Thus, multiplying both sides of~\eqref{eq: weak u} by~$t^m$ yields
\begin{multline*}
\iprod{\Mult^m u,v}+\iprod{\kappa\FI^\alpha\Mult^m\nabla u,\nabla v}
	+\sum_{j=1}^m\tilde d_j^{m,\alpha}
	\iprod{\kappa\FI^{\alpha+m-j}\Mult^j\nabla u),\nabla v}\\
	=\iprod{\Mult^m B^\alpha_{\vec F}u+\Mult^m B^1_{\vec G}u,\nabla v}
	-\iprod{\Mult^m\B^\alpha_a u+\Mult^m B^1_b u,v}
	+\iprod{\Mult^m f,v}
\end{multline*}
for $v\in H^1_0(\Omega)$.  We have
\[
\partial_t^m \FI^{\alpha+m-j}\Mult^j\nabla u
    =\partial_t^j\partial_t^{m-j}\FI^{m-j}\FI^\alpha\Mult^j\nabla u
    =\partial_t^j\FI^\alpha\Mult^j\nabla u
    =\FI^\alpha\partial_t^j\Mult^j\nabla u,
\]
where the final step follows by \eqref{eq: FI omega}  because
\[
\partial_t^i (\Mult^j\nabla u)(0)=0\quad\text{for~$0\le i\le j-1\le m-1$.}
\]
Likewise, $\partial_t^m\FI^\alpha\Mult^m\nabla u
=\FI^\alpha\partial_t^m\Mult^m\nabla u$ because
\[
\partial_t^j(\Mult^m\nabla u)(0)=0\quad\text{for~$0\le j\le m-1$,}
\]
and therefore
\begin{multline}\label{eq: dervi m}
\iprod{\partial_t^m\Mult^m u,v}
	+\iprod{\kappa\FI^\alpha\partial_t^m\Mult^m\nabla u,\nabla v}=
	\iprod{B^{\alpha,m}_{\vec F}u+B^{1,m}_{\vec G}u,\nabla v}\\
	-\iprod{B^{\alpha,m}_a u+B^{1,m}_b u,v}
-\sum_{j=0}^{m-1}\tilde d^{m,\alpha}_j
	\iprod{\kappa\FI^\alpha\partial_t^j\Mult^j \nabla u,\nabla v}
	+\iprod{\partial_t^m\Mult^m f,v}.
\end{multline}
We let $\E(u)=2\|B^{\alpha,m}_{\vec F}u\|^2+\|B^{\alpha,m}_a u\|^2
+2\|B^{1,m}_{\vec G}u\|^2+\|B^{1,m}_b u\|^2$, and conclude using 
the Cauchy--Schwarz inequality that
\begin{multline*}
\iprod{\partial_t^m\Mult^m u,v}
    +\iprod{\kappa\FI^\alpha\partial_t^m\Mult^m\nabla u,\nabla v}
    \le\E(u)+C\sum_{j=0}^{m-1} \|\FI^\alpha\partial_t^j\Mult^j \nabla u\|^2\\
	+\tfrac{1}{2}\|\nabla v\|^2+\tfrac12\|v\|^2+\iprod{\partial_t^m\Mult^m f,v}.
\end{multline*}
Choosing $v=\FI^\alpha\partial_t^m\Mult^m u$, integrating over the
time interval~$(0,t)$ and using \eqref{eq: min eig}, we have
\begin{multline*}
\Q_1^{\alpha,m}(u,t)
	+\tfrac12\Q_2^{\alpha,m}(\nabla u,t)
\le\int_0^t\E(u)\,ds+C\sum_{j=0}^{m-1}\Q_2^{\alpha,j}(\nabla u,t)\\
	+\Q^{\alpha,m}_2(u,t)
+\int_0^t\iprod{\partial_t^m\Mult^m f,\FI^\alpha\partial_t^m\Mult^m u}\,ds,
\end{multline*}
and by the Cauchy-Schwarz inequality and Lemma~\ref{lem: D},
\[
\FI^{1-\alpha}\big(\|\FI^\alpha(\partial_t^m (\Mult^m u))\|^2\big)(t)\le 
2\Q_1^\alpha(\partial_t^m (\Mult^m u),t).
\]
Thus,
\begin{align*}
&\int_0^t\iprod{\partial_t^m\Mult^m f,
	\FI^\alpha\partial_t^m\Mult^m u}\,ds
	\le\int_0^t\|\partial_t^m\Mult^m f\|\,
	\|\FI^\alpha\partial_t^m\Mult^m u\|\,ds\\
&\le C\biggl(\int_0^t(t-s)^\alpha\|\partial_t^m\Mult^m f\|^2\,ds\biggr)^{1/2}
	\biggl(\int_0^t(t-s)^{-\alpha}
	\|\FI^\alpha\partial_t^m\Mult^m u\|^2\,ds\biggr)^{1/2}\\
&\le C\Big(t^\alpha\Q^{0,m}(f,t)\Big)^{1/2}
\Big(\FI^{1-\alpha}(\|\FI^\alpha(\partial_t^m\Mult^m u)\|^2)(t)\Big)^{1/2}\\
&\le Ct^\alpha\Q^{0,m}(f,t)+\tfrac12\Q^{\alpha,m}_1(u,t),
\end{align*}
implying that the function
$\mathsf{q_m}(t)=\Q_1^{\alpha,m}(u,t)+\Q_2^{\alpha,m}(\nabla u,t)$ satisfies
\[
\mathsf{q_m}(t)\le2\int_0^t\E(u)\,ds
	+C\sum_{j=0}^{m-1}\Q^{\alpha,j}_2(\nabla u,t)+Ct^\alpha\Q^{0,m}(f,t).
\]
By Lemma \ref{lem: B integrals},
\begin{equation}\label{eq: E1u}
\Q^0(B^{\alpha,m}_{\vec F}u,t)+\Q^0(B^{\alpha,m}_{a}u,t)
	\le C\sum_{j=0}^m\Q^{\alpha,j}_2(u,t)
\end{equation}
and, applying Lemma \ref{lem: E} with $\nu=1$~and $\mu=\alpha$,
\begin{equation}\label{eq: E2u}
\Q^0(B^{1,m}_{\vec G}u,t)+\Q^0(B^{1,m}_b u,t)\le C\sum_{j=0}^m\Q_2^{1,j}(u,t)
	\le Ct^{2(1-\alpha)}\sum_{j=0}^m\Q_2^{\alpha,j}(u,t)
\end{equation}
for $m\ge0$. By combining the above estimates,
\[
\mathsf{q}_m(t)\le C\Q_2^{\alpha,m}(u,t)+ C\sum_{j=0}^{m-1}\mathsf{q}_j(t)
	+Ct^\alpha\Q^{0,m}(f,t).
\]
Consequently, we conclude (recursively) that
\[
\mathsf{q}_m(t)\le C\sum_{j=0}^m\Q_2^{\alpha,j}(u,t)
	+Ct^{\alpha}\sum_{j=0}^m\Q^{0,j}(f,t),
\]
so, by applying the first inequality in Lemma \ref{lem: D} with~$\phi=(\Mult^j u)^{(j)}$,
\[
\mathsf{q}_m(t)\le Ct^\alpha\sum_{j=0}^m\Q^{0,j}(f,t)
	+C\sum_{j=0}^m\int_0^t\omega_\alpha(t-s)\mathsf{q}_j(s)\,ds.
\]
Therefore,  a repeated application of Lemma~\ref{lem: Gronwall} yields the first
desired estimate.

To show the second estimate, choose $v=\partial_t^m\Mult^m u$
in~\eqref{eq: dervi m} and obtain
\begin{multline*}
\|\partial_t^m\Mult^m u\|^2+\iprod{\kappa\FI^\alpha\partial_t^m\Mult^m \nabla u,
	\partial_t^m\Mult^m\nabla u}
	=-\iprod{Eu,\partial_t^m\Mult^m u}\\
	-\sum_{j=0}^{m-1}\tilde d^{m,\alpha}_j
		\iprod{\kappa\FI^\alpha\partial_t^j\Mult^j \nabla u,
			\partial_t^m\Mult^m\nabla u}
	+\iprod{\partial_t^m\Mult^m f,\partial_t^m\Mult^m u},
\end{multline*}
where $Eu=\nabla\cdot B_{\vec F}^{\alpha,m}u+B_{a}^{\alpha,m}u
+\nabla\cdot B_{\vec G}^{1,m}u+B^{1,m}_b u$. The first and the last
terms on the right-hand side are bounded by
$\|Eu\|^2+\|\partial_t^m\Mult^m f\|^2+\tfrac{1}{2}\|\partial_t^m\Mult^m u\|^2$
so, after integrating in time, using \eqref{eq: min eig} and applying 
\eqref{eq: A} (for a sufficiently \emph{large}~$\epsilon$),
\begin{multline*}
\tfrac12\Q^{0,m}(u,t)+\Q^{\alpha,m}_1(\nabla u,t)
	\le\int_0^t\|Eu(s)\|^2\,ds+\Q^{0,m}(f,t)
	+\tfrac12\Q^{\alpha,m}_1(\nabla u,t)\\
	+C\sum_{j=0}^{m-1}\Q^{\alpha,j}_1(\nabla u,t).
\end{multline*}
Since $\nabla\cdot(\vec F\partial_t^{1-\alpha}u)
=(\nabla\cdot\vec F)\partial_t^{1-\alpha}u
+\vec F\cdot\nabla\partial_t^{1-\alpha}u$, we see that
\[
\nabla\cdot B^{\alpha,m}_{\vec F}u=\partial_t^m\Mult^m\FI^1\bigl(
\nabla\cdot(\vec F\partial_t^{1-\alpha}u)\bigr)
	=B^{\alpha,m}_{\nabla\cdot\vec F}u+B^{\alpha,m}_{\vec F\cdot{}}\nabla u,
\]
and therefore, applying Lemma \ref{lem: B integrals} followed by
Lemma \ref{lem: E},
\begin{align*}
\int_0^t\|Eu(s)\|^2\,ds&\le4\Bigl(\Q^0(\nabla\cdot B_{\vec F}^{\alpha,m}u,t)
	+\Q^0(B_{a}^{\alpha,m}u,t)\\
	&\qquad{}+\Q^0(\nabla\cdot B_{\vec G}^{1,m}u,t)+\Q^0(B^{1,m}_b u,t)\Bigr)\\
	&\le C\sum_{j=0}^m\Bigl(\Q^{\alpha,j}_2(u,t)
		+\Q^{\alpha,j}_2(\nabla u,t)\Bigr).
\end{align*}
Hence, the function $\mathsf{q}_m(t)=\Q^{0,m}(u,t)+\Q^{\alpha,m}_1(\nabla u,t)$
satisfies
\[
\mathsf{q}_m(t)\le2\Q^{0,m}(f,t)
    +C\sum_{j=0}^{m-1}\Q^{\alpha,j}_1(\nabla u, t)
	+C\sum_{j=0}^m\Bigl(\Q^{\alpha,j}_2(u,t)
		+\Q^{\alpha,j}_2(\nabla u,t)\Bigr),
\]
and so, using \eqref{eq: B}~and \eqref{eq: C}, it follows that
\[
\mathsf{q}_m(t)\le2\Q^{0,m}(f,t)+C\sum_{j=0}^{m-1}\mathsf{q}_j(t)
	+C\Bigl(\Q^{\alpha,m}_2(u,t)+\Q^{\alpha,m}_2(\nabla u,t)\Bigr).
\]
By the first inequality in Lemma \ref{lem: D} and \eqref{eq: C},
\[
\Q^{\alpha,m}_2(u,t)+\Q^{\alpha,m}_2(\nabla u,t)
	\le C\int_0^t\omega_\alpha(t-s)\mathsf{q}_m(s)\,ds,
\]
and thus by Lemma \ref{lem: Gronwall},
\[
\mathsf{q}_m(t)\le C\Q^{0,m}(f,t)+C\sum_{j=0}^{m-1}\mathsf{q}_j(t).
\]
Applying this inequality recursively gives
\[
\mathsf{q}_m(t)\le C\sum_{j=0}^m\Q^{0,j}(f,t),
\]
which completes the proof.
\end{proof}

We can now show pointwise bounds for the norms in~$L_2(\Omega)$ of the time
derivatives of $u$~and $\nabla u$.

\begin{theorem}\label{thm: reg}
For $m\ge1$ and $0<t\le T$,
\[
\|(\partial_t^m u)(t)\|^2+t^\alpha\|(\partial_t^m\nabla u)(t)\|^2\le
Ct^{-1-2m}\sum_{j=0}^{m+1}\Q^{0,j}(f,t).
\]
\end{theorem}
\begin{proof}
Since $\Mult\partial_t^m=\partial_t^m\Mult-m\partial_t^{m-1}$, we see using
\eqref{eq: m-fold 2} (and setting $\tilde b^{m,m-1}_m=0$) that
\[
\Mult^{m+1}\partial_t^m=\Mult^m\partial_t^m\Mult-m\Mult^m\partial_t^{m-1}
	=\sum_{j=1}^{m+1}\bigl(\tilde b^{m,m}_{j-1}-m\tilde b^{m,m-1}_{j-1}\bigr)
	\partial_t^{j-1}\Mult^j,
\]
and hence
\begin{equation}\label{eq: Mult m+1 p m}
\|(\Mult^{m+1}\partial_t^mu)(t)\|^2\le C\sum_{j=1}^{m+1}
	\|(\partial_t^{j-1}\Mult^ju)(t)\|^2.
\end{equation}
Using the second inequality in Lemma \ref{lem: D} with~$\phi=\partial_t^{j-1}\Mult^j u$
and the first bound in Lemma \ref{lem: stability C}, we get
\[
\|(\partial_t^{j-1}\Mult^ju)(t)\|^2
	\le Ct^{1-\alpha}\Q^\alpha_1(\partial_t^j\Mult^ju,t)
	\le Ct\sum_{\ell=0}^j\Q^{0,\ell}(f,t)
\]
and so
\[
\|(\partial_t^mu)(t)\|^2=t^{-2m-2}\|(\Mult^{m+1}\partial_t^mu)(t)\|^2
	\le Ct^{-1-2m}\sum_{j=0}^{m+1}\Q^{0,j}(f,t).
\]
Applying the same argument to~$\nabla u$ in place of~$u$, and using the second
bound in Lemma \ref{lem: stability C}, the result follows.
\end{proof}

Next, we estimate \emph{fractional} time derivatives of $u$~and $\nabla u$.  
These bounds will later help in our study of spatial regularity, and reflect the
presence of the fractional time derivative in~\eqref{eq: FPDE}.

\begin{theorem}\label{thm: frac reg}
For $m\ge1$ and $0<t\le T$,
\[
\|(\partial_t^{m-\alpha}u)(t)\|^2
	+t^{\alpha}\|(\partial_t^{m-\alpha}\nabla u)(t)\|^2
	\le Ct^{-1-2(m-\alpha)}\sum_{j=0}^{m+1}\Q^{0,j}(f,t).
\]
\end{theorem}
\begin{proof}
Using the inequality~\eqref{eq: Mult m+1 p m},
\begin{equation}\label{eq: frac reg A}
\begin{aligned}
\bigl\|(\Mult^m\partial_t^{m-\alpha}u)(t)\bigr\|^2
	&=\bigl\|(\Mult^m\partial_t^{m-1}\partial_t^{1-\alpha}u)(t)\bigr\|^2
	\le C\sum_{j=1}^m\bigl\|(
		\partial_t^{j-1}\Mult^j\partial_t^{1-\alpha}u)(t)\bigr\|^2,
\end{aligned}
\end{equation}
and using \eqref{eq: commutator}~and \eqref{eq: FI omega} (with $m=1$),
\begin{align*}
\Mult\partial_t^{1-\alpha}u&=\Mult\partial_t\FI^\alpha u
	=\Mult\bigl(\FI^\alpha\partial_tu+u(0)\omega_\alpha\bigr)\\
	&=(\FI^\alpha\Mult+\alpha\FI^{\alpha+1})\partial_tu+u(0)\Mult\omega_\alpha\\
	&=\FI^\alpha\Mult u'+\alpha\FI^\alpha\bigl(u-u(0)\bigr)
	+\alpha u(0)\omega_{1+\alpha}
	=\FI^\alpha(\Mult u'+\alpha u).
\end{align*}
Thus, by \eqref{eq: m-fold 4},
\[
\Mult^j\partial_t^{1-\alpha}u=\Mult^{j-1}\FI^{\alpha}(\Mult u'+\alpha u)
	=\sum_{\ell=0}^{j-1}\tilde d^{j-1,\alpha}_\ell
	\FI^{\alpha+j-1-\ell}\Mult^\ell(\Mult u'+\alpha u).
\]
We have
\begin{multline*}
\partial_t^{j-1}\FI^{\alpha+j-1-\ell}\Mult^\ell(\Mult u'+\alpha u)
	=\partial_t^\ell\bigl(\partial_t^{j-1-\ell}\FI^{j-1-\ell}\bigr)
		\FI^\alpha\Mult^\ell(\Mult u'+\alpha u)\\
	=\partial_t^\ell\FI^\alpha\Mult^\ell(\Mult u'+\alpha u)
	=\FI^\alpha\partial_t^\ell\Mult^\ell(\Mult u'+\alpha u),
\end{multline*}
where we used the identity~\eqref{eq: FI omega} (with $m=1$) and the fact
that $\partial_t^i\Mult^\ell(\Mult u'+\alpha u)(0)=0$
for $0\le i\le\ell-1$.  Hence,
\[\bigl\|(\partial_t^{j-1}\Mult^j\partial_t^{1-\alpha}u)(t)\bigr\|
	=\biggl\|\sum_{\ell=0}^{j-1}\tilde d^{j-1,\alpha}_\ell
	\FI^{\alpha}\partial_t^\ell\Mult^\ell(\Mult u'+\alpha u)(t)\biggr\|
	\le C\sum_{\ell=0}^{j-1}\|\FI^\alpha\phi_\ell(t)\|\]
where $\phi_\ell=\partial_t^\ell\Mult^\ell(\Mult u'+\alpha u)$.
Using \eqref{eq: m-fold 1},
\[
\phi_\ell
	=\sum_{i=0}^\ell\tilde a^{\ell,\ell}_i\Mult^i\partial_t^i(\Mult u'+\alpha u)
	=\sum_{i=0}^\ell\tilde a^{\ell,\ell}_i\Mult^i
	\bigl(\Mult\partial_t^iu'+i\partial_t^{i-1}u'+\alpha\partial_t^i u\bigr)
\]
and so, by Theorem~\ref{thm: reg},
\begin{equation}\label{eq: frac reg B}
\|\phi_\ell(t)\|^2
	\le C\sum_{r=0}^{\ell+1}\|(\Mult^r\partial_t^ru)(t)\|^2
	\le C\sum_{r=0}^{\ell+1}t^{-1}\sum_{i=0}^{r+1}\Q^{0,i}(f,t)
	\le Ct^{-1}\sum_{r=0}^{\ell+2}\Q^{0,r}(f,t).
\end{equation}
Since $\|\phi_\ell(t)\|\le C\omega_{1/2}(t)\psi_\ell(t)$ where
$\psi_\ell(t)=\sqrt{\sum_{r=0}^{\ell+2}\Q^{0,r}(f,t)}$ is nondecreasing,
we see that $\|\FI^\alpha\phi_\ell(t)\|\le C\omega_{\alpha+1/2}(t)\psi_\ell(t)$.
Therefore,
\begin{align*}
\bigl\|(\partial_t^{j-1}\Mult^j\partial_t^{1-\alpha}u)(t)\bigr\|^2
	&\le C\sum_{\ell=0}^{j-1}\|\FI^\alpha\phi_\ell(t)\|^2
	\le C\sum_{\ell=0}^{j-1}(t^{(\alpha+1/2)-1})^2\psi_\ell(t)^2\\
	&\le Ct^{2\alpha-1}\sum_{\ell=0}^{j+1}\Q^{0,\ell}(f,t),
\end{align*}
and the desired bound for $\|(\partial_t^{m-\alpha}u)(t)\|^2$ follows at once
from~\eqref{eq: frac reg A}.

Replacing $u$ with~$\nabla u$ in the preceding argument, we have
\[\bigl\|(\Mult^m\partial_t^{m-\alpha}\nabla u)(t)\bigr\|^2
	\le C\sum_{j=1}^m\bigl\|(
		\partial_t^{j-1}\Mult^j\partial_t^{1-\alpha}\nabla u)(t)\bigr\|^2
	\le C\sum_{j=1}^m\sum_{\ell=0}^{j-1}\|\FI^\alpha\phi_\ell(t)\|^2\]
where, this time,
$\phi_\ell=\partial_t^\ell\Mult^\ell(\Mult\nabla u'+\alpha\nabla u)$ and hence
\[
\|\phi_\ell(t)\|\le C\omega_{(1-\alpha)/2}(t)\psi_\ell(t).  
\]
It follows that
$\|\FI^\alpha\phi_\ell\|\le C\omega_{(1+\alpha)/2}(t)\psi_\ell(t)$ and therefore
$t^\alpha\bigl\|(\Mult^m\partial_t^{m-\alpha}\nabla u)(t)\bigr\|^2$ is bounded
by
\[
Ct^\alpha\sum_{\ell=0}^{m-1}\|\FI^\alpha\phi_\ell(t)\|^2
	\le Ct^\alpha\sum_{\ell=0}^{m-1}(t^{(1+\alpha)/2-1})^2\psi_\ell(t)^2
	\le Ct^{2\alpha-1}\sum_{\ell=0}^{m+1}\Q^{0,\ell}(f,t),
\]
as required.
\end{proof}

The following simplified bounds are perhaps more immediately useful.

\begin{corollary}\label{cor: reg simple}
Let $m\ge1$ and suppose that $g:(0,T]\to L_2(\Omega)$ is $C^m$ with
\begin{equation}\label{eq: g Cm}
\|g^{(j)}(t)\|\le Mt^{\eta-1-j}\quad\text{for~$0\le j\le m$ and some $\eta>0$.}
\end{equation}
Then 
\[
\|(\partial_t^mu)(t)\|+t^{\alpha/2}\|(\partial_t^m\nabla u)(t)\|
	\le Ct^{-m}(\|u_0\|+Mt^\eta)
\]
and
\[
\|(\partial_t^{m-\alpha}u)(t)\|
	+t^{\alpha/2}\|(\partial_t^{m-\alpha}\nabla u)(t)\|
	\le Ct^{\alpha-m}(\|u_0\|+Mt^\eta).
\]
\end{corollary}
\begin{proof}
Since $\|f^{(j)}(t)\|=\|g^{(j-1)}(t)\|\le Mt^{\eta-j}$ for~$1\le j\le m+1$,
\eqref{eq: m-fold 1} implies that
\[
\text{$\|(\Mult^j f)^{(j)}(t)\|\le CMt^\eta$ for $~1\le j\le m+1$,}
\]
with $\|f(t)\|\le\|u_0\|+M\eta^{-1}t^\eta$. Thus,
\[
\text{$\Q^{0,j}(f,t)\le CM^2t^{2\eta+1}$ for $1\le j\le m+1$,}
\]
with $\Q^0(f,t)\le Ct(\|u_0\|+Mt^\eta)^2$, so
\[
t^{-1-2m}\sum_{j=0}^{m+1}\Q^{0,j}(f,t)
	\le Ct^{-2m}(\|u_0\|+Mt^\eta)^2
\]
and the result follows from Theorems~\ref{thm: reg}~and \ref{thm: frac reg}.
\end{proof}

%%%%%%%%%%%%%%%%%%%%%%%%%%%%%%%%%%%%%%%%%%%%%%%%%%%%%%%%%%%%%%%%%%%%%%%%%%%%%%%
\section{More regular initial data}\label{sec: H2}
We will now investigate further the relation between the regularity 
of~$u$ and that of the initial data~$u_0$.  In particular, 
Theorem~\ref{thm: H2 reg} below extends Corollary~\ref{cor: reg simple} and 
proves a bound used in an error analysis of a finite element discretization of 
the fractional Fokker--Planck equation~\cite{LeMcLeanMustapha2018}. The 
fractional PDE~\eqref{eq: FPDE} can be rewritten as
\[
u'-\nabla\cdot(\kappa\partial_t^{1-\alpha}\nabla u)=h
	\quad\text{for $x\in\Omega$ and $0<t<T$,}
\]
where $h=g-\nabla\cdot\bigl(\vec F\partial_t^{1-\alpha}u+\vec Gu\bigr)
-(a\partial_t^{1-\alpha}u+b u)$. We can therefore apply known results for the 
fractional diffusion equation to establish the following bounds in the norm
$\|v\|_\mu=\|A^{\mu/2}v\|$ of the fractional Sobolev space~$\dot H^\mu(\Omega)$, 
where $A^{\mu/2}$ is defined via the spectral representation 
of~$Av=-\nabla\cdot(\kappa\nabla v)$ using the Dirichlet eigenfunctions 
on~$\Omega$~\cite{McLean2010,Thomee2006}.  The results of this section require 
$H^2$-regularity for the Poisson problem, and to ensure this property we make 
the additional assumptions~\cite[Theorems 2.2.2.3~and 3.2.1.2]{Grisvard2011}
\begin{equation}\label{eq: H2 kappa Omega}
%\kappa\in W^1_\infty\bigl(\Omega;\R^{d\times d}\bigr)
\text{$\kappa$ is Lipschitz on~$\overline\Omega$}
\qquad\text{and}\qquad\text{$\Omega$ is $C^{1,1}$ or convex.}
\end{equation}
It follows that
$\dot H^1(\Omega)=H^1_0(\Omega)$~and
$\dot H^2(\Omega)=H^2(\Omega)\cap H^1_0(\Omega)$.  We also require that $g$ 
satisfies \eqref{eq: g Cm}.  Our first result does not assume any additional 
smoothness of~$u_0$.

\begin{theorem}\label{thm: smoothing}
Assume \eqref{eq: g Cm} ~and \eqref{eq: H2 kappa Omega}.  If 
$u_0\in L_2(\Omega)$, then
\[
t^m\|u^{(m)}(t)\|_\mu\le C\|u_0\|t^{-\mu\alpha/2}+CMt^{\eta-\mu\alpha/2}
	\quad\text{for $0\le\mu\le2$ and $0<t\le T$.}
\]
\end{theorem}
\begin{proof}
We have \cite[Theorems 4.1 and 4.2, and the inequality stated after Theorem
5.4]{McLean2010}
\[
t^m \|u^{(m)}(t)\|_\mu\le C t^{-\mu\alpha/2}\|u_0\|
	+C\sum_{j=0}^m \int_0^t(t-s)^{-\mu\alpha/2}s^j \|h^{(j)}(s)\|\,ds
\]
for $m\ge 0$ and for $0\le \mu\le 2$, with
\begin{multline*}
\|h^{(j)}(s)\|\le\|\partial_t^j g(s)\|
	+C\sum_{\ell=0}^j\Big(\|\partial_t^{\ell+1-\alpha}\nabla u(s)\|
		+\|\partial_t^\ell\nabla u(s)\|\\
	+\|\partial_t^{\ell+1-\alpha}u(s)\|+\|\partial_t^\ell u(s)\|\Big).
\end{multline*}
Corollary \ref{cor: reg simple} shows that $\|h^{(j)}(s)\|$ is bounded by
\[
Ms^{\eta-1-j}+C\sum_{\ell=0}^j\Bigl(
	s^{\alpha/2-\ell-1}+s^{-\alpha/2-\ell}+s^{\alpha-\ell-1}+s^{-\ell}\Bigr)
	(\|u_0\|+Ms^{\eta})
\]
so $s^j\|h^{(j)}(s)\|\le C\|u_0\|s^{\alpha/2-1}+Ms^{\eta-1}$ and hence
\begin{align*}
\int_0^t(t-s)^{-\mu\alpha/2}s^j \|h^{(j)}(s)\|\,ds
	&\le C\|u_0\|(\omega_{1-\mu\alpha/2}*\omega_{\alpha/2})(t)+CM(\omega_{1-\mu\alpha/2}*\omega_{\eta})(t)\\
	&\le C\bigl(\|u_0\|t^{(1-\mu)\alpha/2}+Mt^{\eta-\mu\alpha/2}\bigr),
\end{align*}
completing the proof.
\end{proof}

Recall from part~5 of Theorem~\ref{thm: existence} that for 
any~$u_0\in L_2(\Omega)$ the solution~$u(t)$ converges weakly 
to~$u(0)=u_0$ in~$L_2(\Omega)$ as~$t\to0$. The first estimate in our next 
result shows that $u(t)\to u_0$ in the norm of~$L_2(\Omega)$ if we impose some 
additional spatial regularity on the initial data, namely if $u_0\in\dot 
H^\mu(\Omega)$ for some $\mu>0$.  The second and third estimates extend the 
results of Corollary \ref{cor: reg simple}.

\begin{theorem}\label{thm: u cts}
Assume \eqref{eq: g Cm}~and \eqref{eq: H2 kappa Omega}.  If 
$0\le\mu\le2$~and $u_0\in\dot H^\mu(\Omega)$, then
\[
\|u(t)-u_0\|+t^{\alpha/2}\bigl\|\nabla\bigl(u(t)-u_0\bigr)\bigr\|
	\le C\|u_0\|_\mu t^{\alpha\mu/2}+Mt^\eta,
\]
and, for~$m\ge1$,
\[
\|u^{(m)}(t)\|+t^{\alpha/2}\|\nabla u^{(m)}(t)\|
	\le Ct^{-m}\bigl(\|u_0\|_\mu t^{\alpha\mu/2}+Mt^\eta\bigr)
\]
with
\[
\|\partial_t^{m-\alpha}u(t)\|
	+t^{\alpha/2}\|(\partial_t^{m-\alpha}\nabla u)(t)\|
	\le Ct^{\alpha-m}\bigl(\|u_0\|_\mu t^{\alpha\mu/2}+Mt^\eta\bigr).
\]
\end{theorem}
\begin{proof}
Introduce the solution operator $u(t)=\mathcal{U}(u_0,g,t)$.  By linearity,
$u=u_1+u_2$ where $u_1(t)=\mathcal{U}(u_0,0,t)$~and
$u_2(t)=\mathcal{U}(0,g,t)$.  In view of Corollary \ref{cor: reg simple}, it
suffices to consider~$u_1$.  Let $w(t)=u_1(t)-u_0$ so that $w(0)=0$, and
suppose to begin with that $u_0\in\dot H^2(\Omega)$. Using~\eqref{eq: weak u},
we find that
\[
\iprod{w(t),v}+\bigiprod{\kappa(\FI^\alpha\nabla w)(t),\nabla v}
    -\bigiprod{(\vec B_1w)(t),\nabla v}+\iprod{(B_2w)(t),v}=\iprod{\rho(t),v},
\]
where $\rho(t)=\FI^\alpha\nabla\cdot(\kappa\nabla u_0)
-\nabla\cdot\vec B_1u_0-\vec B_2u_0$. Since 
$(\FI^\alpha u_0)'(t)=u_0\omega_\alpha(t)$, and recalling the definitions
\eqref{eq: vec B1 B2}, we have
\[
\rho'(t)=\bigl(\nabla\cdot(\kappa\nabla u_0)
	-\nabla\cdot(\vec F(t) u_0)-a(t)u_0\bigr)\omega_\alpha(t)
	-\nabla\cdot(\vec G(t)u_0)-b(t)u_0,
\]
so $\|\rho^{(j+1)}(t)\|\le C\|u_0\|_2t^{\alpha-j}$. Therefore, by
Corollary \ref{cor: reg simple},
\[
\|w^{(m)}(t)\|+t^{\alpha/2}\|\nabla w^{(m)}(t)\|
	\le Ct^{-m}\bigl(\|w(0)\|+\|u_0\|_2t^\alpha\bigr)=C\|u_0\|_2t^{\alpha-m},
\]
which proves the result for integer-order time derivatives in the case~$\mu=2$.
Similarly, for the fractional-order time derivatives,
\[\|\partial_t^{m-\alpha}w(t)\|+t^{\alpha/2}\|\partial_t^{m-\alpha}\nabla w(t)\|
\le Ct^{\alpha-m}\bigl(\|w(0)\|+\|u_0\|_2t^\alpha\bigr)
  =Ct^{2\alpha-m}\|u_0\|_2,\]
completing the proof for~$\mu=2$. Since Corollary \ref{cor: reg simple} also
implies the case~$\mu=0$, the result follows for $0<\mu<2$ by interpolation.
\end{proof}

With the help of Lemma \ref{lem: FI psi partial phi} from the Appendix, we can
generalize Theorem~\ref{thm: smoothing} as follows.

\begin{theorem}\label{thm: H2 reg}
Assume \eqref{eq: g Cm} ~and \eqref{eq: H2 kappa Omega}.  If 
$0\le\mu\le2$~and $u_0\in\dot H^\mu(\Omega)$, then
\[
\|u^{(m)}(t)\|_2\le Ct^{-m}\bigl(\|u_0\|_\mu t^{-(2-\mu)\alpha/2}
	+Mt^{\eta-\alpha}\bigr)\quad\text{for $0<t\le T$.}
\]
\end{theorem}
\begin{proof}
We know from Theorem\ref{thm: smoothing} that
\begin{equation}\label{eq: H2 reg A}
\|u^{(m)}(t)\|_2\le Ct^{-\alpha-m}\bigl(\|u_0\|+Mt^\eta\bigl),
\end{equation}
so there is nothing to prove if $u_0=0$.  Thus, by linearity, we may assume
that $g(t)\equiv0$ and so~$M=0$. Integrating~\eqref{eq: FPDE} in time, we see
that
\[
u-\nabla\cdot(\kappa\nabla\FI^\alpha u)+\nabla\cdot\vec B_1u+B_2u=u_0
\]
and, after applying the operator~$\partial_t\FI^{1-\alpha}$ to both sides,
\[
-\nabla\cdot(\kappa\nabla u)=\rho\quad\text{where}\quad
\rho=\partial_t\FI^{1-\alpha}\bigl(u_0-u-\nabla\cdot\vec B_1u -B_2u).
\]
Since $-\nabla\cdot(\kappa\nabla u^{(m)})=\rho^{(m)}$ in~$\Omega$, with 
$u^{(m)}(t)=0$ on~$\partial\Omega$ for~$0\le t\le T$, it follows by 
$H^2$-regularity for the Poisson problem
that
\begin{equation}\label{eq: H2 reg B}
\|u^{(m)}(t)\|_2\le C\|\rho^{(m)}(t)\|.
\end{equation}
The identity~\eqref{eq: FI omega} (with $m=1$) implies that
\begin{align*}
\rho&=\FI^{1-\alpha}\partial_t\bigl(u_0-u-\nabla\cdot\vec B_1u-B_2u)\\
	&=-\FI^{1-\alpha}u'
		-\FI^{1-\alpha}\bigl(\nabla\cdot(\vec F\partial_t^{1-\alpha}u+\vec Gu)
		+a\partial_t^{1-\alpha}u+bu\bigr),
\end{align*}
and Lemma \ref{lem: FI psi partial phi}~and Theorem~\ref{thm: u cts} (with 
$\mu=2$) imply that
\begin{align*}
t^{m+1}\|\partial_t^m\FI^{1-\alpha}u'(t)\|
	&\le C\max_{0\le s\le t}\sum_{j=0}^m s^{1-\alpha+1+j}\|u^{(j+1)}(s)\|\\
	&\le C\max_{0\le s\le t}s^{1-\alpha}
		\bigl(\|u_0\|_2s^\alpha\bigr)= Ct\|u_0\|_2.
\end{align*}
Since $\nabla\cdot(\vec F\partial_t^{1-\alpha}u)
=(\nabla\cdot\vec F)\partial_t^{1-\alpha}u
+\vec F\cdot\partial_t^{1-\alpha}\nabla u$, we see from
Lemma~\ref{lem: FI psi partial phi} and Theorem~\ref{thm: u cts} that
\begin{align*}
t^{m+1}\bigl\|\partial_t^m\FI^{1-\alpha}&\bigl(
\nabla\cdot(\vec F\partial_t^{1-\alpha}u)\bigr)(t)\bigr\|\\
	&\le C\max_{0\le s\le t}\sum_{j=0}^m s^{1-\alpha+1+j}\Bigl(
	\|\partial_s^j(\partial_s^{1-\alpha}u)\|
	+\|\partial_s^j(\partial_s^{1-\alpha}\nabla u)\|\Bigr)\\
	&=C\max_{0\le s\le t}\sum_{j=0}^m s^{1-\alpha+1+j}\Bigl(
	\|\partial_s^{j+1-\alpha}u\|+\|\partial_s^{j+1-\alpha}\nabla u\|\Bigr)\\
	&\le C\max_{0\le s\le t}s^{1-\alpha}(s^\alpha+s^{\alpha/2})
	\bigl(\|u_0\|_2s^\alpha\bigr)\le Ct^{1+\alpha/2}\|u_0\|_2.
\end{align*}
Similarly,
$t^{m+1}\bigl\|\partial_t^m\FI^{1-\alpha}(a\partial_t^{1-\alpha}u)\bigr\|
\le Ct^{1+\alpha/2}\|u_0\|_2$, whereas
\begin{align*}
t^{m+1}\bigl\|\partial_t^m\FI^{1-\alpha}&\bigl(
\nabla\cdot(\vec Gu)\bigr)(t)\bigr\|\\
	&\le C\max_{0\le s\le t}\biggl(
	s^{2-\alpha}\bigl(\|u(s)\|+\|\nabla(u-u_0)(s)\|+\|\nabla u_0\|\bigr)\\
	&\qquad{}+\sum_{j=1}^m s^{1-\alpha+1+j}\bigl(
		\|u^{(j)}(s)\|+\|\nabla u^{(j)}(s)\|\bigr)\biggr)\\
	&\le C\max_{0\le s\le t}\Bigl(s^{2-\alpha}\bigl[\|u_0\|
		+s^{-\alpha/2}(\|u_0\|_2s^\alpha)+\|\nabla u_0\|\bigr]\\
	&\qquad{}+s^{2-\alpha}(1+s^{-\alpha/2})
		\bigl(\|u_0\|_2s^\alpha\bigr)\Bigr)\le Ct^{2-\alpha}\|u_0\|_2
\end{align*}
and $t^{m+1}\bigl\|\partial_t^m\FI^{1-\alpha}(bu)\bigr\|
\le Ct^{2-\alpha}\|u_0\|_2$.  Thus,
\[
t^{m+1}\|\rho^{(m)}(t)\| 
    \le C\bigl(t+t^{1+\alpha/2}+t^{2-\alpha}\bigr)\|u_0\|_2,
\]
showing that $t^m\|\rho^{(m)}(t)\|\le C\|u_0\|_2$
and therefore, by \eqref{eq: H2 reg A}~and \eqref{eq: H2 reg B},
\[
\|u^{(m)}(t)\|_2\le Ct^{-m-\alpha}\|u_0\|
\quad\text{and}\quad
\|u^{(m)}(t)\|_2\le t^{-m}\|u_0\|_2,
\]
which proves the result in the cases $\mu=0$~and $\mu=2$.  The case~$0<\mu<2$ 
then
follows by interpolation.
\end{proof}
%%%%%%%%%%%%%%%%%%%%%%%%%%%%%%%%%%%%%%%%%%%%%%%%%%%%%%%%%%%%%%%%%%%%%%
\appendix
\section{Further technical lemmas}\label{sec: tech}

The following identity uses the notation from Lemma \ref{lem: m-fold}.  

\begin{lemma}\label{lem: q derivative}
Let $\mu>0$ and $1\le q\le m$. If, for $m-q+1\le j\le m$,
\[
\Mult^j\phi\in W^{j-(m-q)}_1(0,T)
\]
with 
\[
(\partial_t^k\Mult^j\phi)(0)=0\quad\text{for $0\le k\le j-(m-q)-1$,}
\]
then
\[
\partial_t^q\Mult^m\FI^\mu\phi
	=\sum_{j=0}^{m-q}\tilde d^{m,\mu}_j\FI^{\mu+m-q-j}\Mult^j\phi
	+\sum_{j=m-q+1}^m\tilde d^{m,\mu}_j\FI^\mu\partial_t^{j-(m-q)}\Mult^j\phi.
\]
\end{lemma}
\begin{proof}
By \eqref{eq: m-fold 4},
\[
\Mult^m\FI^\mu
	=\sum_{j=0}^{m-q} \tilde d^{m,\mu}_j\FI^{\mu+m-j}\Mult^j
	+\sum_{j=m-q+1}^m \tilde d^{m,\mu}_j\FI^{\mu+m-j}\Mult^j.
\]
If $0\le j\le m-q$, then $m-q-j\ge0$ so $\partial_t^q\FI^{\mu+m-j}
=\partial_t^q\FI^q\FI^{\mu+m-q-j}=\FI^{\mu+m-q-j}$. Therefore, 
\[
\partial_t^q\sum_{j=0}^{m-q} \tilde d^{m,\mu}_j\FI^{\mu+m-j}\Mult^j\phi
	=\sum_{j=0}^{m-q} \tilde d^{m,\mu}_j\FI^{\mu+m-q-j}\Mult^j\phi
\quad\text{for $\phi\in L_1(0,T)$.}
\]

If $m-q+1\le j\le m$ then $j-(m-q)\ge1$ so
\[
\partial_t^q\FI^{\mu+m-j}
	=\partial_t^{q-(m-j)}\partial_t^{m-j}\FI^{m-j} \FI^\mu
	=\partial_t^{j-(m-q)}\FI^\mu
\]
and thus
\[
\partial_t^q\sum_{j=m-q+1}^m \tilde d^{m,\mu}_j\FI^{\mu+m-j}\Mult^j\phi
	=\sum_{j=m-q+1}^m\tilde d^{m,\mu}_j\partial_t^{j-(m-q)}\FI^\mu\Mult^j\phi.
\]
By \eqref{eq: FI omega},
\[
\partial_t^{j-(m-q)}\FI^\mu\Mult^j\phi=\FI^\mu\partial_t^{j-(m-q)}\Mult^j\phi
	+\sum_{k=0}^{j-(m-q)-1}(\partial_t^k\Mult^j\phi)(0)\,\omega_{\mu-k},
\]
and our hypotheses on~$\phi$ ensure that all terms in the sum over~$k$ vanish.
 \end{proof}

The next lemma was used in the proof of Lemma \ref{lem: stability C}.

\begin{lemma}\label{lem: B integrals}
Let $\psi\in W^{2m-1}_\infty\bigl((0,T);L_\infty(\Omega)^d\bigr)$ 
for some~$m\ge1$ and let $\mu\ge0$. Then,
\[
\Q^{0,m}(B^\mu_\psi\phi,t)\le C\sum_{j=0}^m\Q^{\mu,j}_2(\phi,t)
\quad\text{for $0\le t\le T$ and $\phi\in C^m_\alpha$.}
\]
\end{lemma}
\begin{proof}
We integrate by parts $m$~times to obtain
\[
B^\mu_\psi\phi=\FI^1(\psi\partial_t^{1-\mu}\phi)
	=\sum_{i=0}^{m-1}(-1)^i\psi^{(i)}\FI^{\mu+i}\phi
	+(-1)^m\FI^1\bigl(\psi^{(m)}\FI^{\mu+m-1}\phi\bigr),
\]
and so
\begin{equation}\label{eq: B integrals A}
B^{\mu,m}_\psi\phi=(\Mult^mB^\mu_\psi\phi)^{(m)}
	=\sum_{i=0}^m(-1)^i\B^m_i\phi,
\end{equation}
where
\[
\B^m_i\phi=\begin{cases}
\partial_t^m\Mult^m\bigl(\psi^{(i)}\FI^{\mu+i}\phi\bigr)
	&\text{for $0\le i\le m-1$,}\\
\partial_t^m\Mult^m\FI^1\bigl(\psi^{(m)}\FI^{\mu+m-1}\phi\bigr)
	&\text{for $i=m$.}
\end{cases}
\]
If $0\le i\le m-1$, then
\[
\B^m_i\phi=\partial_t^m\bigl(\psi^{(i)}(\Mult^m\FI^{\mu+i}\phi)\bigr)
	=\sum_{q=0}^m\binom{m}{q}\psi^{(i+m-q)}\partial_t^q\Mult^m\FI^{\mu+i}\phi
\]
so our assumption on~$\psi$ implies that
\begin{equation}\label{eq: B integrals B}
\|(\B^m_i\phi)(t)\|\le C\sum_{q=0}^m
	\bigl\|\partial_t^q(\Mult^m\FI^{\mu+i}\phi)(t)\bigr\|.
\end{equation}
By Lemma \ref{lem: q derivative},
\begin{multline*}
\partial_t^q\Mult^m\FI^{\mu+i}\phi
	=\sum_{j=0}^{m-q}\tilde d^{m,\mu+i}_j\FI^{\mu+i+m-q-j}\Mult^j\phi\\
	+\sum_{j=m-q+1}^m\tilde d^{m,\mu+i}_j
		\FI^{\mu+i}\partial_t^{j-(m-q)}\Mult^j\phi,
\end{multline*}
and by \eqref{eq: m-fold 3},
\[
\FI^\mu\Mult^j=\sum_{k=0}^jc^{j,\mu}_k\Mult^{j-k}\FI^{\mu+k}
    \quad\text{for $\mu>0$ and $j\ge0$,}
\]
with
\[
\partial_t^q\Mult^j=\sum_{r=0}^qa^{j,q}_r\Mult^{j-r}\partial_t^{q-r}
	=\sum_{r=0}^q\tilde a^{j,q}_r\Mult^{j-q+r}\partial_t^r
	\quad\text{for $1\le q\le j$.}
\]
Thus, for $0\le j\le m-q$,
\[
\FI^{\mu+i+m-q-j}\Mult^j\phi
    =\sum_{k=0}^jc^{j,\mu+i+m-q-j}_k\Mult^{j-k}\FI^{\mu+i+m-q-j+k}\phi
\]
and for $m-q+1\le j\le m$,
\begin{multline*}
\FI^{\mu+i}\partial_t^{j-(m-q)}\Mult^j\phi
=\sum_{r=0}^{j-(m-q)}\tilde a^{j,j-(m-q)}_r\FI^{\mu+i}\Mult^{m-q}
	\Mult^r\partial_t^r\phi\\
=\sum_{r=0}^{j-(m-q)}\tilde a^{j,j-(m-q)}_r\sum_{k=0}^{m-q}c^{m-q,\mu+i}_k
            \Mult^{m-q-k}\FI^{\mu+i+k}\Mult^r\partial_t^r\phi,
\end{multline*}
so
\begin{align*}
\bigl\|(\partial_t^q&\Mult^m\FI^{\mu+i}\phi)(t)\bigr\|^2
    \le C\sum_{j=0}^{m-q}\bigl\|(\FI^{\mu+i+m-q-j}\Mult^j\phi)(t)\|^2\\
    &\qquad{}+C\sum_{j=m-q+1}^m
    \bigl\|(\FI^{\mu+i}\partial_t^{j-(m-q)}\Mult^j\phi)(t)\bigr\|^2\\
    &\le C\sum_{j=0}^{m-q}\sum_{k=0}^j
    \bigl\|(\Mult^{j-k}\FI^{\mu+i+m-q-j+k}\phi)(t)\bigr\|^2\\
    &\qquad+C\sum_{j=m-q+1}^m\sum_{r=0}^{j-(m-q)}\sum_{k=0}^{m-q}
    \bigl\|(\Mult^{m-q-k}\FI^{\mu+i+k}\Mult^r\partial_t^r\phi)(t)\bigr\|^2.
\end{align*}
Integrating in time, since
$\Q^0(\Mult^j\FI^\mu\phi,t)\le t^{2j}\Q^\mu_2(\phi,t)$, we see that
\begin{multline*}
\Q^0\bigl(\partial_t^q\Mult^m\FI^{\mu+i}\phi, t\bigr)
    \le C\sum_{j=0}^{m-q}\sum_{k=0}^j t^{2(j-k)}\Q^{\mu+i+m-q-j+k}_2(\phi,t)\\
    +C\sum_{j=m-q+1}^m\sum_{r=0}^{j-(m-q)}\sum_{k=0}^{m-q}
        t^{2(m-q-k)}\Q^{\mu+i+k,r}_2(\phi,t)
\end{multline*}
and therefore, by Lemma \ref{lem: E},
\begin{equation}\label{eq: B integrals C}
\Q^0\bigl(\partial_t^q\Mult^m\FI^{\mu+i}\phi, t\bigr)
    \le Ct^{2(i+m-q)}\sum_{r=0}^q\Q^{\mu,r}_2(\phi,t).
\end{equation}
Hence, recalling \eqref{eq: B integrals B},
\begin{equation}\label{eq: B integrals D}
\Q^0(\B^m_i\phi,t)\le C\sum_{q=0}^m
    \Q^0\bigl(\partial_t^q\Mult^m\FI^{\mu+i}\phi, t\bigr)
    \le Ct^{2i}\sum_{r=0}^m\Q^{\mu,r}_2(\phi,t),~~{\rm for}~~0\le i\le m-1.
\end{equation}

It remains to estimate
$\B^m_m\phi=\partial_t^m\Mult^m\FI^1(\psi^{(m)}\FI^{\mu+m-1}\phi)$.
Taking $q=m$~and $\mu=1$ in Lemma \ref{lem: q derivative} gives
\[
\partial_t^m\Mult^m\FI^1=\tilde d^{m,1}_0\FI^1+\sum_{j=1}^m\tilde d^{m,1}_j
    \FI^1\partial_t^j\Mult^j,
\]
and so
\[
\B^m_m\phi=\tilde d^{1,m}_0\FI^1(\psi^{(m)}\FI^{\mu+m-1}\phi)
	+\sum_{j=1}^m\tilde d^{1,m}_j\partial_t^{j-1}
		(\psi^{(m)}\Mult^j\FI^{\mu+m-1}\phi).
\]
Thus,
\begin{multline*}
\Q^0(\B^m_m\phi,t)\le C\Q^0\bigl(\FI^1(\psi^{(m)}\FI^{\mu+m-1}\phi),t\bigr)\\
    +C\sum_{j=1}^m\sum_{q=0}^{j-1}
    \Q^0\bigl(\partial_t^q\Mult^j\FI^{\mu+m-1}\phi,t\bigr),
\end{multline*}
and since
\begin{align*}
\bigl\|\bigl(\FI^1(\psi^{(m)}\FI^{\mu+m-1}\phi\bigr)(t)\bigr\|^2
    &\le\biggl(\int_0^t\|\psi^{(m)}(s)\|^2\,ds\biggr)
    \biggl(\int_0^t\|\FI^{\mu+m-1}\phi)(s)\|^2\,ds\biggr)\\
    &\le Ct\Q^{\mu+m-1}_2(\phi,t)
\end{align*}
we have, by Lemma \ref{lem: E},  
\[
\Q^0\bigl(\FI^1(\psi^{(m)}\FI^{\mu+m-1}\phi),t\bigr)
\le Ct^2\Q^{\mu+m-1}_2(\phi,t)\le Ct^{2m}\Q^\mu_2(\phi,t).
\]
Finally, using \eqref{eq: B integrals C} with~$m$ replaced
by~$j$ and with~$i$ replaced by $m-1$,
\begin{align*}
\Q^0(\B^m_m\phi,t)&\le Ct^{2m}\Q^\mu_2(\phi,t)
+C\sum_{j=1}^m\sum_{q=0}^{j-1}t^{2(m-1+j-q)}\sum_{r=0}^q\Q^{\mu,r}_2(\phi,t)\\
    &\le Ct^{2m}\sum_{q=0}^{m-1}\Q^{\mu,r}_2(\phi,t).
\end{align*}
The result now follows from \eqref{eq: B integrals A}~and
\eqref{eq: B integrals D}.
\end{proof}

We used the following result in the proof of Theorem~\ref{thm: H2 reg}.

\begin{lemma}\label{lem: FI psi partial phi}
If $m\ge0$, $\psi\in W^m_\infty\bigl((0,T);L_\infty(\Omega)^d\bigr)$ and
\[
\Mult^k\phi\in W^k_1(0,T)\quad\text{for $0\le k\le m+1$,}
\]
with
\[
(\partial_t^q\Mult^k\phi)(0)=0
	\quad\text{for $1\le q\le k-1$ and $1\le k\le m+1$,}
\]
then
\[
t^{m+1}\bigl\|\partial_t^m\FI^\mu(\psi\phi)(t)\bigr\|
	\le C\max_{0\le s\le t}\sum_{j=0}^m\|s^{\mu+1+j}\phi^{(j)}(s)\|
	\quad\text{for $0<t\le T$.}
\]
\end{lemma}
\begin{proof}
By \eqref{eq: m-fold 2},
\begin{equation}\label{eq: FI psi partial phi A}
\begin{aligned}
\bigl\|\Mult^{m+1}\partial_t^m\FI^\mu(\psi\phi)\|
&=\biggl\|\sum_{j=0}^m\tilde b^{m+1,m}_j\partial_t^j\Mult^{1+j}\FI^\mu(\psi\phi)
	\biggr\|\\
&\le C\sum_{j=0}^m\bigl\|\partial_t^j\Mult^{j+1}\FI^\mu(\psi\phi)(t)\bigr\|,
\end{aligned}
\end{equation}
and in turn,
\[
\partial_t^j\Mult^{j+1}\FI^\mu(\psi\phi)=\sum_{k=0}^{j+1}\tilde
	d^{j+1,\mu}_k\partial_t^j\FI^{\mu+j+1-k}\Mult^k(\psi\phi).
\]
Since $\partial_t^j\FI^{\mu+j+1-k}
=\partial_t^j(\partial_t\FI^1)\FI^{\mu+j+1-k}
=\partial_t^k(\partial_t^{j+1-k}\FI^{j+1-k})\FI^{\mu+1}
=\partial_t^k\FI^{\mu+1}$ for~$0\le k\le j+1$,
\begin{multline}\label{eq: FI psi partial phi B}
\|\partial_t^j\Mult^{j+1}\FI^\mu(\psi\phi)(t)\|
	\le C\sum_{k=0}^{j+1}\|\partial_t^j\FI^{\mu+j+1-k}\Mult^k(\psi\phi)\|\\
	=C\sum_{k=0}^{j+1}\|\partial_t^k\FI^{\mu+1}\Mult^k(\psi\phi)(t)\|
	=C\sum_{k=0}^{j+1}\|\FI^{\mu+1}\partial_t^k\Mult^k(\psi\phi)(t)\|
\end{multline}
where, in the last step, we used the fact that
$\partial_t^q\Mult^k(\psi\phi)(0)=0$ for~$0\le q\le k-1$.  We have
\begin{align*}
\partial_t^k\Mult^k(\psi\phi)&=\partial_t^k(\psi\Mult^k\phi)
	=\sum_{q=0}^k\binom{k}{q}\psi^{(k-q)}\partial_t^q\Mult^k\phi\\
	&=\sum_{q=0}^k\binom{k}{q}\psi^{(k-q)}
	\sum_{r=0}^q\tilde a^{k,q}_r\Mult^{k-(q-r)}\partial_t^r\phi,
\end{align*}
and hence
\begin{align*}
\|\FI^\mu\partial_t^k\Mult^k(\psi\phi)(t)\|
	&\le C\sum_{q=0}^k\sum_{r=0}^q
	\|(\FI^{\mu+1}\Mult^{k-q+r}\partial_t^r\phi)(t)\|\\
	&=C\sum_{q=0}^k\sum_{r=0}^q\int_0^t\omega_{\mu+1}(t-s)s^{k-q-\mu-1}
		\|s^{r+\mu+1}\phi^{(r)}(s)\|\,ds\\
	&\le C\sum_{r=0}^k\biggl(\max_{0\le s\le t}\|s^{r+\mu+1}\phi^{(r)}(s)\|
		\biggr)\sum_{q=r}^k(\omega_{\mu+1}*\omega_{k-q-\mu})(t).
\end{align*}
Since $(\omega_{\mu+1}*\omega_{k-q-\mu})(t)=\omega_{k-q+1}(t)\le Ct^{k-q}$.
the result now follows from \eqref{eq: FI psi partial phi A}~and
\eqref{eq: FI psi partial phi B}.
\end{proof}
%%%%%%%%%%%%%%%%%%%%%%%%%%%%%%%%%%%%%%%%%%%%%%%%%%%%%%%%%%%%%%%%%%%%%%%%%%%%%%%
\bibliographystyle{elsarticle-num}
\bibliography{regularity_ref}
%%%%%%%%%%%%%%%%%%%%%%%%%%%%%%%%%%%%%%%%%%%%%%%%%%%%%%%%%%%%%%%%%%%%%%%%%%%%%%%
\end{document}